\documentclass[11 pt]{amsart}

\usepackage{lineno,hyperref}
\usepackage{amsmath,amssymb}
\usepackage{enumitem}
\usepackage{lineno}
\newtheorem{thm}{Theorem}[section]
\newtheorem{defn}[thm]{Definition}
\newtheorem{prop}[thm]{Proposition}
\newtheorem{cor}[thm]{Corollary}
\newtheorem{lem}[thm]{Lemma}

\begin{document}


\title{Characterizations of Weaving $K$-Frames}

\author[A. Bhandari]{Animesh Bhandari $^\dagger$}

\address{$^\dagger$Department of Mathematics\\ NIT Meghalaya\\ Shillong 793003\\ India}
\email{animesh@nitm.ac.in}

\author[D. Borah]{Debajit Borah {$^{\dagger}$}}
\address{$^\dagger$Department of Mathematics\\ NIT Meghalaya\\ Shillong 793003\\ India}
\email{debajitborah439@gmail.com}

\author[S. Mukherjee]{Saikat Mukherjee {$^{\dagger *}$}}
\address{$^\dagger$Department of Mathematics\\ NIT Meghalaya\\ Shillong 793003\\ India}
\email{saikat.mukherjee@nitm.ac.in}

$\thanks{*Corresponding author}$


\begin{abstract}
In distributed signal processing frames play significant role as redundant building blocks. Bemrose et al. were motivated from this concept, as a result they introduced weaving frames in Hilbert space. Weaving frames have useful applications in sensor networks, likewise weaving $K$-frames have been proved to be useful during signal reconstructions from the range of a bounded linear operator $K$. This article focuses on study, characterization of weaving $K$-frames in different spaces. Paley-Wiener type perturbations and conditions on erasure of frame components have been assembled to scrutinize woven-ness of $K$-frames.\\

\end{abstract}

\maketitle

\section{Introduction}\label{Sec-Preli}
The concept of Hilbert frames was first introduced by Duffin and Schaeffer \cite{DuSc52} in 1952. After several decades, in 1986, frame theory has been popularized by the groundbreaking work by Daubechies, Grossman and Meyer \cite{DaGrMa86}.
Since then frame theory has been widely used by mathematicians and engineers in various fields of mathematics and engineering, namely, signal processing \cite{Fe99}, sensor networks \cite{CaKuLiRo07}, etc.

Also frame theory literature became popularized through several generalizations, likewise, fusion frame (frames of subspaces) \cite{CaKu04} , $G$-frame (generalized frames) \cite{Su06}, $K$-frame (atomic systems) \cite{Ga12}, $K$-fusion frame (atomic subspaces) \cite{Bh17}, etc. and these generalizations have been proved to be useful in various applications.






For detail discussion regarding frames, readers are referred to the books \cite{CaKu12, Ch08}.

Throughout the manuscript, $\mathcal{H}$ is denoted as separable Hilbert space with inner product, $\langle \cdot, \cdot \rangle$, and associated norm, $\|\cdot\|$, on it. We denote by $\mathcal{L}(\mathcal{H}_1, \mathcal{H}_2)$ the space of all bounded linear operators from  $\mathcal{H}_1$ into $\mathcal{H}_2$,  $\mathcal L(\mathcal H)$ for $\mathcal L(\mathcal H, \mathcal H)$ and $R(K)$ as range of the operator $K$. Further, $T^{\dag}$ denotes the Moore-Penrose pseudo inverse of $T$ and $P_M$ denotes the orthogonal projection on $M$. We also use the notations $[m]$ for the set $\{1, 2, \cdots , m\}$ and $\mathcal I$ for finite or countable index set.

\subsection{Frame} A collection  $\{ f_i \}_{i\in \mathcal I}$ in $\mathcal{H}$ is called a \emph{frame} if there exist constants $A,B >0$ such that \begin{equation}\label{Eq:Frame} A\|f\|^2~ \leq ~\sum_{i\in \mathcal  I} |\langle f,f_i\rangle|^2 ~\leq ~B\|f\|^2,\end{equation}for all $f \in \mathcal{H}$. The numbers $A, B$ are called \emph{frame bounds}. The supremum over all $A$'s and infimum over all $B$'s satisfying above inequality are called the \emph{optimal frame bounds}.
If a collection satisfies only the right inequality in equation (\ref{Eq:Frame}), it is called a {\it Bessel sequence}.

Given a frame $\{f_i\}_{i\in \mathcal  I}$ for  $\mathcal H$, the corresponding synthesis operator is a bounded linear operator $T: l^2(\mathcal  I)\rightarrow \mathcal H$ and is defined by $T\{c_i\} =  \sum \limits_{i\in \mathcal  I} c_i f_i$. The adjoint of $T$, $T^*: \mathcal{H} \rightarrow l^2(\mathcal  I)$, given by $T^*f = \{\langle f, f_i\rangle\}_{i \in \mathcal  I}$, is called the analysis operator. The frame operator, $S=TT^*: \mathcal H \rightarrow \mathcal H$, is defined by $$Sf=TT^*f = \sum_{i\in \mathcal  I} \langle f, f_i\rangle f_i.$$ It is well-known that the frame operator is bounded, positive, self-adjoint and invertible.



\subsection{$K$-Frame}


There are several generalizations of frame, all of these generalizations have been proved to be useful in many applications. In the sequel, we discuss results on one such generalization of frame, called {\bf $ K$-frame}. The notion of $ K$-frames was introduced by L. G\v{a}vru\c{t}a in \cite{Ga12} to study the {\it atomic systems} with respect to a bounded linear operator $K$ in $\mathcal{H}$.

\begin {defn} ({\bf $K$-Frame})
Let $K \in \mathcal L(\mathcal H)$. A collection  $\{ f_i \}_{i \in \mathcal I}$ in $\mathcal{H}$ is called a \emph{ $K$-frame} if there exist constants $A,B >0$ such that \begin{equation}A\|K^*f\|^2~ \leq ~\sum_{i \in \mathcal I} |\langle f,f_i \rangle|^2 ~\leq ~B\|f\|^2,\end{equation} for all $f \in \mathcal{H}$. The numbers $A,B$ are called \emph{$K$-frame bounds}. The above collection is said to be a {\it tight $K$-frame} if
\begin{equation}\label{Eq:K-tight-fram}
A\|K^* f\|^2 = \sum_{i\in \mathcal I} | \langle f,f_i \rangle |^2, \end{equation}
for all $f\in \mathcal H$.
\end{defn}

$K$-frames are more general than ordinary frames in the sense that the lower frame bound only holds for the elements in the range of $K$. Because of the higher generality of $K$-frames, the associated $K$-frame operator need not be invertible.




\subsection{{Woven and $K$-Woven Frame}} In general in a sensor networking system, a frame can be characterized by signals. If there are two frames, having same characteristics, then in absence of a frame element from the first frame, still we are able to get an error free result on account of the replacement of the frame element of first frame by the frame element of second frame.

In this context basically one can think of the  intertwinedness between two sets of sensors, or in general between two frames, which leads to the idea of weaving frames. Weaving frames or woven frames were introduced by Bemrose et al. in \cite{BeCaGrLaLy16}. Later the concept of woven-ness has been characterized by Bhandari et al. in \cite{BhMu19} and characterization of weaving $K$-frames has been produced by Deepshikha et al. in \cite{DeVa18}.


\begin{defn}
	In $\mathcal H$, two frames $\lbrace f_i \rbrace_{i \in \mathcal I}$ and  $\lbrace g_i \rbrace_{i \in \mathcal I}$ are said to be {\it woven} if for every $\sigma \subset \mathcal I$,  $\lbrace f_i \rbrace_{i \in \sigma} \cup \lbrace g_i \rbrace_{i \in \sigma^c}$ also forms a frame for $\mathcal H$ and the associated frame operator for every weaving is defined as \cite{BhMu19},
	$$S_{\mathcal F \mathcal G}f=\sum_{i \in \sigma}\langle f,f_i \rangle f_i + \sum_{i \in \sigma^c}\langle f,g_i \rangle g_i, ~~ \text{for}~ \text{all}~f \in \mathcal H.$$
\end{defn}

\begin{defn} \cite{DeVa18} A family of K-frames $\{\{\phi_{ij}\}_{j=1}^\infty: i \in[m]\} $ for $\mathcal{H}$ is said to be K-woven if there exist universal positive constants $ A, B $ such that for any partition $\{\sigma_i\}_{i\in[m]}$ of $\mathcal{I}$, the family $\bigcup_{i \in [m]} \{\phi_{ij}\}_{j\in\{\sigma_i\}}$ is a K-frame for $\mathcal{H}$ with lower and upper K-frame bounds A and B, respectively. Each family $\bigcup_{i \in [m]} \{\phi_{ij}\}_{j\in \sigma_i}$ is called a $K$-weaving.
\end{defn}

The following result discuss the woven-ness of Bessel sequences by means of the associated synthesis operator.

\begin{prop} \cite {BeCaGrLaLy16} \label{Bessel} Let $\{f_{ij}\}_{i \in \mathcal I}$ be a collection of Bessel sequences in $\mathcal H$ with bounds $B_j$'s for every $j \in [m]$, then every weaving forms a Bessel sequence with bound $\sum \limits_{j \in [m]} B_j$ and norm of corresponding synthesis operator is bounded by $\sqrt{\sum \limits_{j \in [m]} B_j}$.
	
\end{prop}

The following Lemma provides a discussion regarding Moore-Penrose  pseudo-inverse. For detail discussion regarding the same we refer \cite{Ch08, Ka80}.


\begin{lem}\label{Moore} Let $\mathcal H$ and $\mathcal K$ be two Hilbert spaces and $T \in \mathcal L(\mathcal H,\mathcal K)$ be a closed range operator, then the followings hold:
	\begin{enumerate}
		\item $TT^{\dag}=P_{T(\mathcal H)}$, $T^{\dag}T=P_{T^*(\mathcal K)}$
		\item $\frac {\|f\|} {\|T^{\dag}\|} \leq \|T^*f\|$ for all $f \in T(\mathcal H)$.
		\item $TT^{\dag}T=T$, $T^{\dag}TT^{\dag}=T^{\dag}$, $(TT^{\dag})^*=TT^{\dag}$, $(T^{\dag}T)^*=T^{\dag}T$.
	\end{enumerate}
\end{lem}

\section{Main Results}\label{Results}

We begin this section by providing two intertwining results on $K$-frames between two separable Hilbert spaces.
\begin{lem}\label{lem_K_intertwine1}
Let $ K \in \mathcal{L}(\mathcal{H}_1)$,  $ T \in \mathcal{L}(\mathcal{H}_1,\mathcal{H}_2)$, and $\mathcal{F} = \{f_i\}_{i \in \mathcal I} $ be a K-frame for $\mathcal{H}_1$. Then $T\mathrm{\mathcal F} = \{Tf_i\}_{i \in \mathcal{I}} $ is a $TKT^*$- frame for $\mathcal{H}_2$.
\end{lem}
\begin{proof}
	Since $\{f_i\}_{i \in \mathcal{I}}$ is a K-frame for $\mathcal{H}_1$, then there exists $A, B>0$ so that
	\begin{align*}
 A\|K^*h_1\|^2 \leq \sum\limits_{i \in \mathcal I} |\langle h_1,f_i \rangle |^2 \leq B \|h_1\|^2,
 \end{align*}
for every $h_1 \in \mathcal{H}_1 $. Now for every $h_2  \in \mathcal{H}_2$ we obtain,	
\begin{align*}
\sum\limits_{i \in \mathcal I} |\langle h_2,Tf_i \rangle|^2 &  \leq B\|T^*h_2\|^2 \leq   B\|T\|^2 \|h_2\|^2,
\end{align*}
and
\begin{align*}
\frac {A} { \|T\|^2 } \|(TKT^*)^*h_2\|^2 &  \leq A \|K^*(T^*h_2)\|^2 \leq \sum\limits_{i \in \mathcal I} |\langle h_2,Tf_i \rangle|^2.
\end{align*}
Therefore $ T \mathcal F $ is a $ TKT^*$- frame for $\mathcal H_ 2$.
\end{proof}

\begin{lem}\label{lem_K_intertwine2}
 Let $\{f_i\}_{i \in \mathcal I} \subset \mathcal H_1$, $T \in \mathcal L(\mathcal H_1, \mathcal H_2)$ be one-one, closed range operator so that $\{T f_i\}_{i \in \mathcal I}$ is a $K$-frame for $R(T) \subset \mathcal H_2$ for some $K \in \mathcal  L(\mathcal H_2)$. Then $\{f_i\}_{i \in \mathcal I}$ is a $T^{\dag} K T$-frame for $\mathcal H_1$.
\end{lem}

\begin{proof} Since $\{T f_i\}_{i \in \mathcal I}$ is a $K$-frame for $R(T)$, there exist $A, B > 0$ such that for every $h_2 ^{(1)} \in R(T)$ we have,
	\begin{equation} \label{K-frame}
	A \|K^* h_2 ^{(1)}\|^2 \leq \sum \limits_{i \in \mathcal I} |\langle h_2 ^{(1)}, T f_i \rangle|^2 \leq B \|h_2 ^{(1)}\|^2.
	\end{equation}
Now since $T$ is one-one and $R(T)$ is closed, for every $h_1 \in \mathcal H_1$ there exists $h_2 \in \mathcal H_2$ so that $h_1=T^{*} h_2$ and  for every $h_2 \in \mathcal H_2$ we have $h_2 = h_2 ^{(1)} + h_2 ^{(2)}$, where $h_2 ^{(1)} \in R(T)$ and $h_2 ^{(2)} \in R(T)^{\perp}$.

Therefore, $ h_2 ^{(1)} = T^{*\dag}(h_1 - T^*h_2 ^{(2)}) = T^{*\dag} h_1$. Hence using equation (\ref{K-frame}) we obtain,
\begin{eqnarray*}
\sum_{i \in \mathcal I} | \langle h_1,f_i \rangle|^2 = \sum\limits_{i \in \mathcal I} | \langle h_2^{(1)},T f_i \rangle|^2 & \geq &  A \|K^*h_2^{(1)}\|^2 \\
& \geq & \frac {A} {\|T\|^2} \|(T^{\dag} K T)^* h_1\|^2.
	\end{eqnarray*}
Thus the conclusion follows from the following,
$$\sum\limits_{i \in \mathcal I} | \langle h_1, f_i \rangle|^2 = \sum\limits_{i \in \mathcal I} | \langle h_2^{(1)}, T f_i \rangle|^2 \leq B \|T^{\dag}\|^2 \|h_1\|^2.$$
\end{proof}

As a consequence of Lemma \ref{lem_K_intertwine1} and \ref{lem_K_intertwine2}, the following two propositions show that $K$-woven-ness is preserved under bounded linear operators.

\begin{prop} \label{prop_K_intertwine1} Let $K \in \mathcal L(\mathcal H_1)$, $\mathcal{F}=\{f_i\}_{i\in \mathcal{I}}$ and $\mathcal{G}=\{g_i\}_{i\in \mathcal{I}}$ be $K$-frames for $\mathcal{H}_1$ and suppose $T \in \mathcal{L}(\mathcal{H}_1,\mathcal{H}_2)$. If $\mathcal{F}$ and $\mathcal{G}$ are K-woven in $\mathcal{H}_1$, then $T\mathcal{F}$ and $T\mathcal{G}$ are $TKT^*$-woven in $\mathcal{H}_2$.
\end{prop}
\begin{proof}
Applying Lemma \ref{lem_K_intertwine1}, our assertion is tenable.
\end{proof}

\begin{prop}\label{prop_K_intertwine2}
	Suppose $\{f_i\}_{i \in \mathcal I}, \{g_i\}_{i \in \mathcal I} \subset \mathcal H_1$ and $K \in \mathcal  L(\mathcal H_2)$. Consider $T \in \mathcal L(\mathcal H_1, \mathcal H_2)$ to be one-one and $R(T)$ is closed so that $\{T f_i\}_{i \in \mathcal I}$ and $\{T g_i\}_{i \in \mathcal I}$  are $K$-woven in $R(T)$ with the universal bounds $A, B$. Then $\{ f_i\}_{i \in \mathcal I}$  and $\{ g_i\}_{i \in \mathcal I}$ are  $T^{\dag} K T$-woven in $\mathcal H_1$ with the universal bounds $\frac {A} {\|T\|^2}$, $B \|T^{\dag}\|^2$.
\end{prop}

\begin{proof} The proof will be followed from Lemma \ref{Moore} and Lemma \ref{lem_K_intertwine2}.
\end{proof}

In the following result we provide a necessary and sufficient conditions for woven frames by means of $K$-woven frames.

\begin{prop}\label{equivalency1}
	Suppose $K\in \mathcal L(\mathcal H)$. Then the following statements are equivalent:
 \begin{enumerate}
 \item $\{f_i\}_{i \in \mathcal{I}}$ and $\{g_i\}_{i \in \mathcal{I}}$ are woven in $R(K^*)$.

  \item $\{Kf_i\}_{i \in \mathcal{I}}$ and $\{Kg_i\}_{i \in \mathcal{I}}$ are $K$-woven in $\mathcal H$.
  \end{enumerate}
\end{prop}

\begin{proof}
 \noindent \underline {{(1) $\Rightarrow$ (2)}}

	Let $\{f_i\}_{i \in \mathcal{I}}$ and $\{g_i\}_{i \in \mathcal{I}}$ be woven in $R(K^*)$ with the universal bounds $A, B$, then for every $\sigma \subset \mathcal{I}$ and every $f \in R(K^*)$ we have,
\begin{equation} \label{necessary}
A\|f\|^2 \leq \sum\limits_{i \in \sigma} |\langle f, f_i \rangle |^2 + \sum\limits_{i\in {\sigma}^c} | \langle f, g_i \rangle |^2 \leq B\|f\|^2.
\end{equation}
	
Moreover, for every $f \in \mathcal{H}$, $K^*f \in R(K^*)$ and	therefore using equation (\ref{necessary}), for every $\sigma \subset \mathcal{I}$ and for every $f \in \mathcal{H}$ we obtain,
	\begin{eqnarray*}
		\sum \limits _{i \in \sigma} |\langle f, Kf_i \rangle|^2 + \sum\limits_{i \in {\sigma}^c} | \langle f, Kg_i \rangle |^2
 & = &	\sum\limits_{i \in \sigma} |\langle K^* f, f_i \rangle|^2
 \\ &+ & \sum\limits_{i \in {\sigma}^c} | \langle K^* f, g_i \rangle |^2
\\ & \geq & A \|K^* f\|^2.
	\end{eqnarray*}
	The upper bound of the same weaving will be executed by Proposition  \ref{Bessel}.\\
	
 \noindent \underline {{(2) $\Rightarrow$ (1)}} Suppose $\{Kf_i\}_{i \in \mathcal{I}}$ and $\{Kg_i\}_{i \in \mathcal{I}}$ are $K$- woven with the universal bounds $C, D$. Then for every $\sigma \subset \mathcal{I}$ and for every $f \in \mathcal{H}$ we have,
	\begin{equation}\label{converse}
		\sum\limits_{i \in \sigma} |\langle f, Kf_i \rangle|^2 + \sum\limits_{i \in {\sigma}^c} | \langle f, Kg_i \rangle |^2
 \geq  C \|K^* f\|^2.
	\end{equation}
Again for every $g \in R(K^*)$, there exists $f \in \mathcal H$ so that $g=K^*f$ and hence using equation (\ref{converse}), for every $\sigma \subset \mathcal{I}$ and for every $g \in R(K^*)$ we obtain,
\begin{eqnarray*}
 \sum\limits_{i \in \sigma} |\langle g, f_i \rangle|^2 + \sum\limits_{i \in {\sigma}^c} | \langle g, g_i \rangle |^2
& = &  \sum\limits _{i \in \sigma} |\langle f, K f_i \rangle|^2 + \sum\limits_{i \in {\sigma}^c} | \langle f, K g_i \rangle |^2
\\&\geq & C \|K^*f\|^2
\\& = & C \|g\|^2.
\end{eqnarray*}
	
	The upper bound of the same weaving will achieved by Proposition \ref{Bessel}.
\end{proof}

Next result provides a characterization of woven frames through $K$-woven frames.

\begin{prop} \label{equivalency2}
Let $\mathcal F=\{f_i\}_{i \in \mathcal I}, ~\mathcal G=\{g_i\}_{i \in \mathcal I}\subset \mathcal H$ and $K \in \mathcal{L}(\mathcal H)$. Then
\begin{enumerate}
\item $\mathcal F, \mathcal G$ are woven  in $R(K)$ implies they are $K$-woven in $\mathcal H$.

\item $\mathcal F, \mathcal G$ are $K$-woven in $R(K)$ implies they are woven  in $R(K)$, provided $R(K)$ is closed.
\end{enumerate}
\end{prop}

\begin{proof}
\begin{enumerate}
\item Suppose $\{f_i\}_{i \in \mathcal I}$ and $\{g_i\}_{i \in \mathcal I}$ are woven in $R(K)$ with the universal bounds $A, B$. Then for every $\sigma \subset \mathcal I$ and every $f \in \mathcal H$ we get,
\begin{eqnarray*}\frac {A} {\|K\|^2} \|K^* f\|^2 \leq \sum\limits_{i \in \sigma} |\langle f, f_i \rangle |^2 + \sum\limits_{i \in {\sigma}^c} | \langle f, g_i \rangle |^2 \leq B \|f\|^2.
\end{eqnarray*}

\item Let $\{f_i\}_{i \in \mathcal I}$ and $\{g_i\}_{i \in \mathcal I}$ be $K$-woven with the universal bounds $C, D$. Applying closed range property of $K$ (see Lemma \ref{Moore}), for every $f \in R(K)$ we have $\frac {\|f\|^2} {\|K^{\dag}\|^2} \leq \|K^* f\|^2$ and therefore for every $\sigma \subset \mathcal I$ and every $f \in R(K)$ we obtain,
\begin{eqnarray*}\frac {C} {\|K^{\dag}\|^2} \|f\|^2 \leq \sum\limits_{i \in \sigma} |\langle f, f_i \rangle |^2 + \sum\limits_{i \in {\sigma}^c} | \langle f, g_i \rangle |^2 \leq D \|f\|^2.
\end{eqnarray*}
\end{enumerate}
\end{proof}

In the following results we discuss stability of $K$-woven-ness under perturbation and erasures. Analogous erasure result for frame can be observed in \cite{CaLy15}.

\begin{thm} \label{pert}
	Let $T, K \in \mathcal{L}(\mathcal H)$ with $K$ has closed range  and suppose for every $f \in \mathcal H$ we have, $\|(T^* - K^*)f\| \leq \alpha_1 \|T^*f\| + \alpha_2 \|K^*f\| + \alpha_3 \|f\|$, for some $\alpha_1, \alpha_2, \alpha_3  \in (0,1)$. Then $\{f_i\}_{i \in \mathcal{I}}$ and $\{g_i\}_{i \in \mathcal{I}}$ are $T$- woven if  they are $K$- woven, in $R(K)$.
\end{thm}

\begin{proof}
	Let $\{f_i\}_{i \in \mathcal I}$ and $\{g_i\}_{i \in \mathcal I}$ be $K$- woven with the universal bounds $A, B$. Then for every $\sigma \subset \mathcal{I}$ and every $f \in R(K)$ we have,
	\begin{eqnarray} \label{paley}
~~~~~~A\|K^*f\|^2 \leq \sum\limits_{i \in \sigma} |\langle f, f_i \rangle |^2 + \sum\limits_{i \in {\sigma}^c} | \langle f, g_i \rangle |^2 \leq B \|f\|^2.
	\end{eqnarray}
	
	
Since $\|K^*f\| \geq  \|T^*f\| - \|(T^* - K^*)f\|$ for every $f\in \mathcal H$, applying closed range property of $K$ (see Lemma \ref{Moore}) and using given perturbation condition for every $f \in R(K)$ we obtain,
	$$(1 - \alpha_1) \|T^*f\| \leq (1 + \alpha_2 + \alpha_3 \|K^{\dag}\|) \|K^*f\|.$$
	Therefore, using  equation (\ref{paley}), for every $f \in R(K)$ and every $\sigma \subset \mathcal{I} $ we obtain,
	\begin{eqnarray*}
 A \left (\frac{1 - \alpha_1} { 1+ \alpha_2 + \alpha_3 \|K^{\dag}\|}\right)^2 \|T^*f\|^2
& \leq &  \sum\limits_{i \in \sigma} |\langle f, f_i \rangle|^2
\\ &+ & \sum\limits_{i \in {\sigma}^c} | \langle f, g_i \rangle|^2
\\ & \leq &  B\|f\|^2.
	\end{eqnarray*}
	
\end{proof}

\begin{cor}
Let $T, K \in \mathcal{L}(\mathcal H)$ and suppose $\alpha_1, \alpha_2 \in (0,1)$ so that for every $f \in \mathcal H$ we have,
$\|T^* f - K^*f\| \leq \alpha_1 \|T^* f\| + \alpha_2 \|K^*f\|$. Then  $\{f_i\}_{i \in \mathcal{I}}$ and $\{g_i\}_{i \in \mathcal{I}}$ are $T$-woven if and only if they are $K$-woven.
\end{cor}

\begin{thm}\label{Thm:erasure1}
Let $\mathcal{F}=\{f_i\}_{i\in \mathcal{I}}$ and $\mathcal{G}=\{g_i\}_{i \in \mathcal{I}}$ be K-woven  in $ \mathcal{H}_1 $ with universal lower bound $A$ and $T \in \mathcal{L}(\mathcal{H}_1,\mathcal{H}_2)$. Let us assume $ \mathcal{J} \subset \mathcal{I}$ and $ 0<C<\frac{A}{\|T\|^2}$ so that for every  $ f \in \mathcal{H}_2 $
	\begin{eqnarray} \label{inequa}
	\sum\limits_{i \in \mathcal J} | \langle f,Tf_i \rangle|^2 \leq C \|TK^*T^*f\|^2.
	\end{eqnarray}
	Then, $\{Tf_i\}_{i\in \mathcal I \setminus \mathcal J}$ and $\{Tg_i\}_{i\in \mathcal I \setminus \mathcal J}$ are $TKT^*$-woven in $\mathcal{H}_2$ .
\end{thm}

\begin{proof} Since $\mathcal{F}, \mathcal{G}$ are K-woven  in $ \mathcal{H}_1 $, then by Lemma \ref{lem_K_intertwine1} and Proposition \ref{prop_K_intertwine1},  $T\mathcal{F}$ and $T\mathcal{G}$ are $TKT^*$-woven with universal lower bound $\frac{A}{\|T\|^2}$ in $\mathcal H_2$. Therefore, applying equation (\ref{inequa}), for every $\sigma \subset \mathcal{I} \setminus \mathcal{J}$ and for every $f \in \mathcal H_2$ we obtain,
	\begin{eqnarray*}
	&& \sum \limits_{i \in \sigma} | \langle f,Tf_i \rangle|^2 + \sum\limits_{i \in \sigma^c} |\langle f,Tg_i \rangle|^2
	\\ & =& \sum\limits_{i \in \sigma \cup \mathcal{J}} | \langle f,Tf_i \rangle|^2 +
 \sum\limits_{i \in \sigma^c} | \langle f,Tg_i \rangle|^2 - \sum\limits_{i \in \mathcal{J}} | \langle f,Tf_i\rangle |^2 \\
		& \geq & \frac{A}{\|T\|^2} \|(TKT^*)^*f\|^2 - C\|(TKT^*)^*f\|^2 \\
		& = & \left(\frac{A}{\|T\|^2}-C\right)\|(TKT^*)^*f\|^2,
	\end{eqnarray*}
	where $\sigma^c$ is the complement of $\sigma$ in $\mathcal I \setminus \mathcal J$.
	
	The universal upper bound will be followed by Proposition \ref{Bessel}.	
\end{proof}

By choosing $\mathcal H_1=\mathcal H_2$ and $T = I$, we obtain the following result.

\begin{cor}
Let $\mathcal{F}=\{f_i\}_{i\in \mathcal{I}}$ and $\mathcal{G}=\{g_i\}_{i \in \mathcal{I}}$ are K-woven in $\mathcal{H}$ with universal bounds A, B. Let us consider $ \mathcal{J} \subset \mathcal{I}$ and $0<C<A$ such that for every $f\in \mathcal{H}$,
	$$ \sum\limits_{i \in \mathcal{J}} | \langle f,f_i\rangle |^2 \leq C \|K^*f\|^2,$$
	then $\{f_i\}_{i \in \mathcal I \setminus \mathcal J}$ and $\{g_i\}_{i \in \mathcal I \setminus \mathcal J}$ are K-frames and also they are K-woven with universal bounds $(A- C), B$.
\end{cor}


Using Proposition \ref{prop_K_intertwine2}, we get the following result analogous to Theorem \ref{Thm:erasure1}.
\begin{thm}\label{Thm:erasure2}
Let $\{f_i\}_{i \in \mathcal I}, \{g_i\}_{i \in \mathcal I} \subset \mathcal H_1$ and $K \in \mathcal  L(\mathcal H_2)$. Suppose $T \in \mathcal L(\mathcal H_1, \mathcal H_2)$ is one-one and $R(T)$ is closed so that $\{T f_i\}_{i \in \mathcal I}$ and $\{T g_i\}_{i \in \mathcal I}$  are $K$-woven in $R(T)$ with the universal lower bound $A$. Further suppose  $ \mathcal{J} \subset \mathcal{I}$ and $ 0<C<\frac{A}{\|T\|^2}$ so that for every  $ f \in \mathcal{H}_1 $
	\begin{eqnarray} \label{inequal}
	\sum\limits_{i \in \mathcal J} | \langle f, f_i \rangle|^2 \leq C \|(T^\dag K T)^*f\|^2.
	\end{eqnarray}
	Then, $\{f_i\}_{i\in \mathcal I \setminus \mathcal J}$ and $\{g_i\}_{i\in \mathcal I \setminus \mathcal J}$ are $T^{\dag} K T$-woven in $\mathcal{H}_1$ .
\end{thm}

	
	

{\bf{Acknowledgment}}

The first author acknowledges the fiscal support of MHRD, Government of India and the third author is supported by DST-SERB project MTR/2017/000797. 


\end{document}